\documentclass[preprint,12pt]{elsarticle}

\usepackage{amsmath,amssymb,amsthm,mathtools}
\usepackage{enumitem}
\usepackage{array}
\usepackage{microtype}
\usepackage[hidelinks]{hyperref}
\usepackage[T1]{fontenc}
\usepackage{lmodern}
\journal{Journal of Combinatorial Theory, Series B}
\biboptions{numbers,sort&compress}

\newtheorem{theorem}{Theorem}[section]
\newtheorem{lemma}[theorem]{Lemma}
\newtheorem{proposition}[theorem]{Proposition}
\newtheorem{corollary}[theorem]{Corollary}
\theoremstyle{definition}
\newtheorem{definition}[theorem]{Definition}
\newtheorem{remark}[theorem]{Remark}

\newcommand{\F}{\mathbb F}
\newcommand{\OA}{\operatorname{OA}}
\newcommand{\GL}{\operatorname{GL}}
\newcommand{\circuittag}[1]{\text{[circuit $#1$]}}

\begin{document}
\begin{frontmatter}

\title{Four-Entropic Matroids Are Quaternary}

\author[sharif]{Mohammad Hossein Kalantari}
\ead{mohammadhosein.kalantari@sharif.edu}
\ead{mohammadhossein.kalantari12@gmail.com}
\author[sharif]{Shahram Khazaei}
\ead{shahram.khazaei@sharif.ir}
\address[sharif]{Department of Mathematical Sciences, Sharif University of Technology, Tehran, Iran}

\begin{abstract}
For an integer $q\ge2$, a matroid is $q$-entropic if its rank function, multiplied by $\log q$, is the joint-entropy function of random variables on a $q$-element alphabet. We prove that a matroid is $4$-entropic if and only if it is representable over $\F_4$. The corresponding statements for alphabet sizes two and three were known. The proof combines minor closure and the excluded-minor characterization of quaternary matroids with structural properties of quasigroups of order four. Thus arbitrary four-symbol partition representations yield no matroids beyond the quaternary ones. As an application, every access structure admitting an ideal perfect scheme with a uniform four-symbol secret and four-symbol active shares also admits an ideal $\F_4$-linear scheme.
\end{abstract}


\end{frontmatter}

\section{Introduction}\label{sec:intro}

A field representation records the dependence structure of a matroid by linear dependence of vectors. Partition representations and almost-affine codes provide a nonlinear fixed-alphabet analogue: a finite alphabet is assigned to each matroid element, and the number of possible joint coordinate values is prescribed by matroid rank. In entropy language, for an integer $q\ge2$ we call a matroid \emph{$q$-entropic} if its rank function, multiplied by $\log q$, is the joint-entropy function of random variables taking values in a common $q$-element alphabet. A matroid representable over $\F_4$ is called \emph{quaternary}.

This fixed-alphabet viewpoint appears in several equivalent forms. Mat\'u\v{s} developed partition representations and the associated matroid quasigroup equations~\cite{Matus1999}; Simonis and Ashikhmin studied almost-affine codes~\cite{SimonisAshikhmin1998}; Abbe and Spirkl investigated fixed-alphabet entropy representations~\cite{AbbeSpirkl2019}; and Chen, Cheng, and Bai related these questions to variable-strength orthogonal arrays and probabilistically characteristic sets~\cite{ChenChengBai2021,ChenChengBai2025}. We use results from these lines of work, but our question is a matroid representation problem: does an arbitrary four-symbol representation produce any matroid that is not quaternary?

Before the entropic-matroid terminology, the binary and ternary rigidity cases were already present in the ideal secret-sharing literature: Beimel and Chor proved that ideality over a binary, respectively ternary, secret domain is characterized by matroids representable over $\F_2$, respectively $\F_3$, a result later recalled by Beimel and Livne~\cite{BeimelChor1994,BeimelLivne2008}; Mat\'u\v{s} also states the equivalent degree-two and degree-three partition-representation theorem~\cite{Matus1999}. Abbe and Spirkl subsequently formulated these rigidity facts in the entropic-matroid setting~\cite{AbbeSpirkl2019}. We prove the corresponding fixed-alphabet rigidity theorem at alphabet size four; in contrast, rigidity fails at size nine, where the non-Pappus matroid has an almost-affine representation although it is representable over no field~\cite{SimonisAshikhmin1998}.

Our main result settles this case.

\begin{theorem}[Degree-four rigidity]\label{thm:main}
For every finite matroid $M$,
\[
M\text{ is $4$-entropic}
\quad\Longleftrightarrow\quad
M\text{ is representable over }\F_4.
\]
\end{theorem}

Degree-four partition representability is minor-closed. The excluded-minor theorem of Geelen, Gerards, and Kapoor therefore reduces the forward implication to seven exclusions. Published results give the two uniform exclusions and the non-Fano exclusion. Sections~\ref{sec:tools}--\ref{sec:structural} establish the remaining four, namely $P_6$, $(F_7^-)^*$, $P_8$, and $P_8''$, from Mat\'u\v{s}'s circuit-quasigroup equations.

A binary quasigroup of order four has an orthogonal mate only in the Klein isotopy class. After the alphabet is identified with $\F_2^2$, every permutation of its four points is affine. These facts turn the selected circuit equations into affine identities. For $P_8$ and $P_8''$, four cube-face relations force four ternary quasigroups to be simultaneously affine in fixed coordinates; the resulting matrix equations yield the exclusions.

\noindent\textbf{Concurrent work.}\quad
Independently and concurrently, Beimel, Ben-Efraim, Farr\`as, and Moya proved,
in the language of ideal secret sharing, the corresponding access-structure
characterization: an access structure is $4$-ideal if and only if it admits an
ideal $\F_4$-linear scheme. Their argument identifies the associated matroid
classes using different techniques, and they also settle domain size
six~\cite{BeimelBenEfraimFarrasMoya2026}.

\noindent\textbf{Paper organization.}\quad
Section~\ref{sec:fixed} develops partition representations and their
quasigroup formulation. Section~\ref{sec:reduction} gives the excluded-minor
proof framework, and Sections~\ref{sec:tools}--\ref{sec:structural} establish
the remaining exclusions. Section~\ref{sec:sss} gives the secret-sharing
consequence, and Section~\ref{sec:verify} describes supplementary independent
verification.

\section{\texorpdfstring{Preliminaries}{Preliminaries}}\label{sec:fixed}

We use standard matroid terminology. Let $M$ be a finite matroid on ground set $E$, with rank function $r$. Its dual is denoted by $M^*$, and $U_{r,n}$ denotes the rank-$r$ uniform matroid on $n$ elements. A \emph{circuit-hyperplane} is a set that is both a circuit and a hyperplane. For $S\subseteq E$, the restriction is denoted by $M|S$; deletion and contraction are denoted by $M\setminus e$ and $M/e$.

\begin{definition}
Let $\mathbb F$ be a field. An \emph{$\mathbb F$-representation}
of a matroid $M=(E,r)$ is a family of vectors $(v_e)_{e\in E}$ in a
finite-dimensional vector space over $\mathbb F$ such that, for every
$A\subseteq E$,
\[
r(A)=\dim_{\mathbb F}\operatorname{span}\{v_e:e\in A\}.
\]
A matroid admitting such a family is \emph{representable over $\mathbb F$}.
In particular, a matroid representable over $\mathbb F_4$ is called
\emph{quaternary}.
\end{definition}

For a finite random variable $X$ with distribution $p$, its Shannon entropy is
\[
H(X)=-\sum_x p(x)\log p(x),
\]
with one logarithm base fixed throughout. For a tuple $(X_e)_{e\in E}$ and $A\subseteq E$, write $X_A=(X_e:e\in A)$.

\begin{definition}
Let $q\ge2$. The matroid $M=(E,r)$ is \emph{$q$-entropic} if there are random variables $(X_e)_{e\in E}$ taking values in a common $q$-element alphabet such that
\[
H(X_A)=r(A)\log q\qquad(A\subseteq E).
\]
\end{definition}

We use Mat\'u\v{s}'s terminology~\cite[Definition~1.1]{Matus1999}.
A partition of a finite set is a family of nonempty, pairwise disjoint subsets
whose union is the whole set; its members are called \emph{blocks}. The common
refinement of two partitions consists of the nonempty intersections of their
blocks. The definition extends in the evident way to any finite family of
partitions.

\begin{definition}
A \emph{p-representation of degree $q$} of $M=(E,r)$ consists of a set $\Omega$ with $|\Omega|=q^{r(E)}$ and partitions $\xi_e$ of $\Omega$, $e\in E$, such that for every $A\subseteq E$ the common refinement of $(\xi_e)_{e\in A}$ has exactly $q^{r(A)}$ blocks, each of cardinality $q^{r(E)-r(A)}$. The set $\Omega$ is called the ground set of the p-representation.
\end{definition}

\begin{proposition}[Entropy--partition equivalence]\label{prop:entropy-partition}
For every integer $q\ge2$, a matroid is $q$-entropic if and only if it has a p-representation of degree $q$.
\end{proposition}

This standard equivalence follows by passing between a uniformly distributed random vector and its level-set partitions; see \cite[Section~3.1.2 and Theorem~1]{ChenChengBai2021}. The same object is an almost-affine representation in the terminology of Simonis--Ashikhmin and a variable-strength orthogonal array in the terminology of Chen--Cheng--Bai~\cite{SimonisAshikhmin1998,ChenChengBai2021}.

Putting p-representations into coordinate form makes their existence and classification easier to analyze. The next definitions introduce this framework.

\begin{definition}
Let $n$ be a positive integer, let $Q$ be a finite set, and let $\Omega\subseteq Q^n$. For each $1\le i\le n$ and $a\in Q$, define
\[
\xi_i(a)=\{x\in\Omega:x_i=a\}.
\]
The nonempty sets among the $\xi_i(a)$ form a partition of $\Omega$, called the $i$th \emph{coordinate partition}.
\end{definition}

If $\xi$ is a p-representation of a matroid $M$ on $\Omega_\xi$ and $f:\Omega_\xi\to\Omega_\eta$ is a bijection, transporting every partition through $f$ gives a p-representation $\eta$ on $\Omega_\eta$. The following notion identifies representations that differ only by such a relabelling~\cite[Definition~1.3]{Matus1999}.

\begin{definition}
	Two p-representations $\xi$ and $\eta$ of a matroid $M = (E,r)$ on ground sets $\Omega_{\xi}$ and $\Omega_{\eta}$, respectively, are \emph{p-isotopic} if 
	there exists a bijection $f: \Omega_{\xi} \rightarrow  \Omega_{\eta}$ such that
	$f\xi_i = \eta_i$ for all $i \in E$. The representations $\xi$ and $\eta$ are then said to be \emph{p-isotopic}.
\end{definition}

\begin{remark}\label{rem: transforming to subset of cartesian}
	Consider a p-representation $\xi$ of degree $q$ for a matroid $M = (E,r)$ on 
	ground set $\Omega_{\xi}$. Fix an arbitrary set $Q$ of cardinality $q$. For each $i\in E$, label the blocks of $\xi_i$ injectively by elements of $Q$, and let $f_i:\Omega_\xi\to Q$ assign to each point the label of its block. The coordinate map $f=(f_i)_{i\in E}:\Omega_\xi\to Q^E$ is injective. Considering the 
	coordinate partitions for $\Omega_{\eta} = f(\Omega_{\xi})$ gives a 
	p-representation $\eta$ of $M$ on the ground set $\Omega_{\eta} \subseteq Q^E$.
	Moreover, $\xi$ and $\eta$ are p-isotopic via $f$.
\end{remark}

The following proposition will be used repeatedly later.

\begin{proposition}\label{prop: coordinate partition p-rep}
	Assume $\xi$ and $\eta$ are p-representations of a matroid $M=(E,r)$ on the ground sets
	$\Omega_{\xi} \subseteq Q^E$ and $\Omega_{\eta} \subseteq R^E$, respectively,
	where $|Q|=|R|$ is their common degree and both representations
	have coordinate partitions. If they are p-isotopic via the map
	$f:\Omega_{\xi} \rightarrow \Omega_{\eta}$, then there is a bijection
	$f_e:Q \rightarrow R$ for each $e \in E$, and on $\Omega_\xi$
	$f = (f_e)_{e \in E}$.
	Conversely, if $\xi$ is a p-representation of the matroid
	$M=(E,r)$ on the ground set $\Omega_{\xi} \subseteq Q^E$, and $f_e:Q \rightarrow R$ is a bijection for each $e \in E$, then using $f=(f_e)_{e \in E}$ to transfer $\xi$ gives another p-representation $\eta$ on the ground set
	$\Omega_{\eta} \subseteq R^E$ with coordinate partitions. 
\end{proposition}

\begin{proof}
For a nonloop $e$, the coordinate partitions $\xi_e$ and
$\eta_e$ each have exactly $|Q|=|R|$ nonempty blocks. Since the p-isotopy
$f$ maps $\xi_e$-blocks bijectively to $\eta_e$-blocks, the $e$th coordinate
of $f(x)$ depends only on $x_e$ and defines a bijection $f_e:Q\to R$. If $e$
is a loop, only one symbol occurs in each $e$th coordinate; map the occurring
symbol in $Q$ to the occurring symbol in $R$ and extend this assignment
arbitrarily to a bijection $Q\to R$. In either case,
$f(x)=(f_e(x_e))_{e\in E}$ for every $x\in\Omega_\xi$.

Conversely, a coordinatewise family of bijections carries every common
refinement block to a common refinement block of the same cardinality.
Therefore all defining block counts of a p-representation are preserved.
\end{proof}

A natural problem is to determine whether a given matroid has a p-representation and, more generally, to classify all such representations up to p-isotopy. Mat\'u\v{s} resolves this problem for several matroids~\cite{Matus1999}.

To prove Theorem~\ref{thm:main}, we must exclude degree-four p-representations of several matroids. Mat\'u\v{s}'s quasigroup formulation is particularly convenient for this purpose.

\begin{definition}
	Suppose $C$ and $Q$ are non-empty finite sets. A subset $\Lambda \subseteq Q^C$
	is called a \emph{(C-)quasigroup} over $Q$, if for every $i \in C$, and every
	 $x_{C\setminus\{i\}}\in Q^{C\setminus\{i\}}$, there is a unique $y\in Q$ such that the vector whose $i$th coordinate is $y$ and whose remaining coordinates are $x_{C\setminus\{i\}}$ belongs to $\Lambda$. Thus $\Lambda$ determines a map $\Lambda^i:Q^{C\setminus\{i\}}\to Q$ by $\Lambda^i(x_{C\setminus\{i\}})=y$.
\end{definition}

\begin{definition}
	Suppose $k$ is a positive integer, and $Q$ is a non-empty finite set. A map
	$f:Q^k \rightarrow Q$ is called a \emph{quasigroup map} if fixing any $k-1$ inputs makes it a bijection of $Q$ in the remaining input.
\end{definition}

\begin{remark}
	The two definitions above are closely related. In fact, if $\Lambda$ is a quasigroup, fixing $i \in C$ gives the map $\Lambda^i$ as defined above. This map is a quasigroup map. Conversely, the graph of a quasigroup map $f:Q^k\to Q$, viewed as a subset of $Q^k\times Q$, is a quasigroup.
\end{remark}

For a matroid $M$, let $\mathcal{C}(M)$ denote its set of circuits; moreover, if $B$ is a basis of $M$ and $j\notin B$, let $\gamma(j;B)$ denote the unique circuit contained in $B\cup\{j\}$, called the fundamental circuit of $j$ with respect to $B$.

The following definition, which is equivalent to p-representability in the sense made precise below, is due to Mat\'u\v{s}~\cite[Definition~2.1]{Matus1999}.

\begin{definition}[Matroid quasigroup equations]\label{def:matroid-quasigroup-equations}
	Let $M=(E,r)$ be a matroid and let $Q$ be a nonempty set. A family
	$(\Gamma_C)_{C\in\mathcal C(M)}$ of $C$-quasigroups on $Q$ is a
	\emph{solution of the matroid quasigroup equations of $M$} if the following
	holds. For every basis $B$, every $i\in E\setminus B$, every circuit $C$
	containing $i$, and every $x_B\in Q^B$,
	\begin{equation}\label{eq:matroid-quasigroup-equations}
		\Gamma_{\gamma(i;B)}^i
		\bigl(x_{\gamma(i;B)\setminus\{i\}}\bigr)
		=
		\Gamma_C^i\bigl((y_j)_{j\in C\setminus\{i\}}\bigr),
	\end{equation}
	where
	\[
	y_j=
	\begin{cases}
		x_j,&j\in C\cap B,\\
		\Gamma_{\gamma(j;B)}^j
		\bigl(x_{\gamma(j;B)\setminus\{j\}}\bigr),
		&j\in C\setminus(B\cup\{i\}).
	\end{cases}
	\]
	Each instance of \eqref{eq:matroid-quasigroup-equations} is called a
	\emph{matroid quasigroup equation}.
\end{definition}

Mat\'u\v{s} also introduced the following equivalence relation on solutions of the matroid quasigroup equations~\cite[Definition~2.3]{Matus1999}.

\begin{definition}[Simultaneous isotopy]\label{def:simultaneous-isotopy}
	Let $(\Gamma_C)_{C\in\mathcal C(M)}$ and
	$(\Lambda_C)_{C\in\mathcal C(M)}$ be families of circuit quasigroups on
	alphabets $Q$ and $R$, respectively. They are \emph{simultaneously isotopic}
	if there are bijections $h_e:Q\to R$, one for each $e\in E$, such that, for
	every circuit $C$ and every $x_C\in Q^C$,
	\[
	x_C\in\Gamma_C
	\quad\Longleftrightarrow\quad
	\bigl(h_j(x_j)\bigr)_{j\in C}\in\Lambda_C.
	\]
\end{definition}

P-representations and solutions of the matroid quasigroup equations are equivalent in the following precise sense~\cite[Proposition~2.4]{Matus1999}.

\begin{proposition}[Mat\'u\v{s}'s correspondence]\label{prop:matus-correspondence}
	Let $q\ge2$. A matroid is p-representable of degree $q$ if and only
	if its matroid quasigroup equations have a solution on a $q$-element
	set. Moreover, there is a one-to-one correspondence between the p-isotopy
	classes of the p-representations of the matroid and the simultaneous 
	isotopy classes of the solutions of the matroid quasigroup equations.
\end{proposition}

\begin{remark}\label{rem:isotopy-invariance}
For a set $C\subseteq E$, let $\pi_C:Q^E\to Q^C$ denote the canonical coordinate projection, defined by
\[
\pi_C\bigl((x_i)_{i\in E}\bigr)=(x_i)_{i\in C}.
\]
Every simultaneous isotope of a solution is again a solution. Moreover, if $\xi$ is a p-representation of degree $q$ for a matroid $M=(E,r)$ on a ground set $\Omega\subseteq Q^E$ with coordinate partitions, then $\Gamma_C=\pi_C(\Omega)$ for $C\in\mathcal C(M)$ gives a family of quasigroups $(\Gamma_C)_{C\in\mathcal C(M)}$ that solves the matroid quasigroup equations of $M$.
\end{remark}

\begin{lemma}[Minor closure]\label{lem:minor}
If $M$ has a partition representation of degree $q$, then every minor of $M$ has a partition representation of degree $q$.
\end{lemma}

\begin{proof}
{For fixed-alphabet entropic representations, minor closure is also proved by
Abbe and Spirkl~\cite[Lemma~2]{AbbeSpirkl2019}. We include the partition proof
to keep the degree $q$ explicit.}
Deletion discards a coordinate partition. For contraction, let $e$ be a nonloop and choose one block $\Omega_0$ of $\xi_e$. Then $|\Omega_0|=q^{r(E)-1}$. Restrict each remaining partition to $\Omega_0$. For $A\subseteq E\setminus\{e\}$, the restricted common refinement has
\[
q^{r(A\cup\{e\})-1}=q^{r_{M/e}(A)}
\]
equal-sized blocks, using $r_{M/e}(A)=r(A\cup\{e\})-1$. Thus the restrictions represent $M/e$. If $e$ is a loop, contraction equals deletion.
\end{proof}

\section{Excluded-minor proof of the main theorem}\label{sec:reduction}

We write $F_7^-$ for the non-Fano matroid and $(F_7^-)^*$ for its dual. We use the following theorem of Geelen, Gerards, and Kapoor~\cite{GeelenGerardsKapoor2000}.

\begin{theorem}[Geelen--Gerards--Kapoor]\label{thm:ggk}
A matroid is representable over $\F_4$ if and only if it has no minor isomorphic to any of
\[
U_{2,6},\quad U_{4,6},\quad P_6,\quad F_7^-,\quad (F_7^-)^*,\quad P_8,\quad P_8''.
\]
\end{theorem}

The reverse implication of Theorem~\ref{thm:main} is standard: evaluating a
uniformly random linear functional on an $\F_4$-representation gives the
required entropic representation. For the forward implication, suppose that
$M$ is $4$-entropic. By Proposition~\ref{prop:entropy-partition}, it has a
partition representation of degree four, and Lemma~\ref{lem:minor} gives such
a representation for every minor. Lemma~\ref{lem:classical} and
Theorems~\ref{lem:p6}, \ref{lem:dnf}, and \ref{lem:p8} exclude all seven
minors in Theorem~\ref{thm:ggk}. Hence $M$ is quaternary, proving
Theorem~\ref{thm:main}.

Table~\ref{tab:matroids} summarizes the excluded minors and the auxiliary
graphic matroid used in the proof. The labelled copies needed for the
structural arguments are specified immediately below and in
Lemma~\ref{lem:matroid isomorphism}.

\begin{table}[ht]
\centering
\caption{Matroids used in the excluded-minor reduction and its proof.}
\label{tab:matroids}
\small
\begin{tabular}{@{}lcc>{\raggedright\arraybackslash}p{0.55\textwidth}@{}}
\hline
Matroid & Rank & Size & Role in the proof \\
\hline
$U_{2,6}$ & $2$ & $6$ & Excluded by the entropy bound in Lemma~\ref{lem:classical}. \\
$U_{4,6}$ & $4$ & $6$ & Excluded by the orthogonal-array bound in Lemma~\ref{lem:classical}. \\
$F_7^-$ & $3$ & $7$ & Excluded by the parity result in Lemma~\ref{lem:classical}. \\
$P_6$ & $3$ & $6$ & Excluded in Theorem~\ref{lem:p6}. \\
$(F_7^-)^*$ & $4$ & $7$ & Excluded in Theorem~\ref{lem:dnf}. \\
$P_8$ & $4$ & $8$ & Excluded in Theorem~\ref{lem:p8}. \\
$P_8''$ & $4$ & $8$ & Excluded in Theorem~\ref{lem:p8}. \\
$M(K_4)$ & $3$ & $6$ & Auxiliary contraction minor identified with $(F_7^-)^*/4$ in Lemma~\ref{lem:matroid isomorphism}. \\
\hline
\end{tabular}
\end{table}

For the structural exclusions we fix labelled copies. We abbreviate $\{0,1,2\}$ by $012$, and similarly for other sets. The rank-three matroid $P_6$ has ground set $\{0,\ldots,5\}$, with $012$ its unique three-element circuit; every other triple is a basis. We take $(F_7^-)^*$ to be the rank-four matroid on $\{0,\ldots,6\}$ whose nonbasis four-sets are the circuit-hyperplanes
\begin{equation}\label{eq:dnf}
0145,\quad0235,\quad0346,\quad1234,\quad1356,\quad2456.
\end{equation}
We take $P_8$ to be the rank-four matroid on $\{0,\ldots,7\}$ whose nonbasis four-sets are
\begin{equation}\label{eq:p8}
\begin{split}
0127,&\ 0136,\ 0235,\ 1234,\ 0456,\\
1457,&\ 2467,\ 3567,\ 0347,\ 1256.
\end{split}
\end{equation}
The last two form the unique disjoint pair of circuit-hyperplanes. The matroid $P_8''$ is obtained by \emph{relaxing} them, meaning that these two sets are also declared to be bases; hence its circuit-hyperplanes are the first eight sets in \eqref{eq:p8}.

Three exclusions are already available in the literature.

\begin{lemma}\label{lem:classical}
None of $U_{2,6}$, $U_{4,6}$, or $F_7^-$ has a partition representation of degree four.
\end{lemma}

\begin{proof}
Abbe and Spirkl prove that $U_{2,q+2}$ is not $q$-entropic for every $q\ge2$~\cite[Lemma~11]{AbbeSpirkl2019}; Proposition~\ref{prop:entropy-partition} gives the $U_{2,6}$ assertion. A degree-four representation of $U_{4,6}$ would be an index-one orthogonal array $\OA(4^4,6,4,4)$. The $t\ge q$ case of Bush's bound gives at most $t+1=5$ columns when $t=q=4$~\cite{Bush1952}. Finally, in the proof of \cite[Proposition~4.1]{Matus1999}, Mat\'u\v{s} shows that every partition representation of the non-Fano matroid $F_7^-$ has odd degree.
\end{proof}

It remains to establish the four degree-four exclusions for $P_6$,
$(F_7^-)^*$, $P_8$, and $P_8''$ invoked above.

\section{\texorpdfstring{Preliminary lemmas}{Preliminary lemmas}}\label{sec:tools}

Put $V = \F_2^2$. This is a set of cardinality $4$. Moreover, it is a vector
space of dimension $2$ over $\F_2$, the $2$-element field. Consider the elements of $V$ as row vectors. A quasigroup map
$f:V^k \rightarrow V$ is called an \emph{affine}  quasigroup map on $V$ if it has the form
\begin{equation}\label{eq:affine}
q(x_1,\ldots,x_k)=x_1A_1+\cdots+x_kA_k+c,
\qquad A_i\in\GL_2(\F_2),\quad c\in V.
\end{equation}

Two binary quasigroups $q,r:V^2\to V$ are \emph{orthogonal} if $(x,y)\mapsto(q(x,y),r(x,y))$ is a bijection of $V^2$. In that case, each is called an orthogonal mate of the other.

\begin{lemma}[Order-four orthogonality]\label{lem:orthogonal4}
If $f$ is a binary quasigroup on $V = \F_2^2$ which has an orthogonal mate, then
it is an affine quasigroup on $V$, i.e., there are $A,B \in \GL_2(\F_2)$, and 
$c \in V$ such that for all $x,y \in V$ 
\[
f(x,y)=xA+yB+c.
\]
\end{lemma}

\begin{proof}
By the order-four classification of Latin squares with an orthogonal
mate~\cite[Section~5]{EganWanless2016}, there are permutations
$\sigma_0,\sigma_1,\sigma_2$ of $V$ such that
\[
\sigma_0\bigl(f(x,y)\bigr)=\sigma_1(x)+\sigma_2(y).
\]
Moreover,
\[
|\operatorname{AGL}_2(\F_2)|=4|\GL_2(\F_2)|=24=|S_4|,
\]
and the affine group acts faithfully on the four points of $V$. Hence every
permutation of $V$ is affine. Writing
$\sigma_i(z)=zL_i+t_i$, where $L_i\in\GL_2(\F_2)$, and solving the displayed
identity for $f(x,y)$ gives
\[
f(x,y)=xA+yB+c
\]
for some $A,B\in\GL_2(\F_2)$ and $c\in V$.
\end{proof}

\begin{lemma}[Affine retracts]\label{lem:retract}
Let $T:V^3\to V$. If every binary retract of $T$, meaning every map obtained by fixing one input, is affine, then $T$ is affine. More precisely,
\[
T(x,y,z)=xA+yB+zC+t
\]
for some $2\times2$ matrices $A,B,C$ and $t\in V$. If $T$ is a ternary quasigroup, then $A,B,C\in\GL_2(\F_2)$.
\end{lemma}

\begin{proof}
For each $z$ write
\[
T(x,y,z)=xA_z+yB_z+c_z,
\]
and for each $y$ write
\[
T(x,y,z)=xC_y+zD_y+e_y.
\]
The affine linear part of a map on $V^2$ is unique, so $A_z=C_y$ for every $y,z$; call the common matrix $A$. Thus $T(x,y,z)-xA$ is independent of $x$. The retract $T(0,y,z)$ is affine in $(y,z)$, hence
\[
T(x,y,z)=xA+yB+zC+t.
\]
If $T$ is a quasigroup, then with the other two variables fixed, each of the maps $x\mapsto xA$, $y\mapsto yB$, and $z\mapsto zC$ must be bijective. Therefore $A,B,C$ are invertible.
\end{proof}

We label $M(K_4)$ so that its triangles are
$124$, $136$, $235$, and $456$.

\begin{lemma}\label{lem:matroid isomorphism}
	Consider the matroid $(F_7^-)^*/4$. The following map gives an isomorphism from
	$(F_7^-)^*/4$ to $M(K_4)$:
\[
0 \mapsto1,\qquad 1 \mapsto 2,\qquad 2 \mapsto 5,\qquad 3 \mapsto 3, \qquad 
5 \mapsto 4, \qquad 6 \mapsto 6.
\]
Moreover, $(F_7^-)^*/4$ has $7$ circuits 
\[
015, \qquad 036, \qquad 123, \qquad 256, \qquad 0126, \qquad 0235,
\qquad 1356.
\]
These circuits are mapped to the following circuits of $M(K_4)$, respectively,
\[
124, \qquad 136, \qquad 235, \qquad 456, \qquad 1256, \qquad 1345,
\qquad 2346.
\]
\end{lemma}

\begin{proof}
Contracting element $4$ in the labelled matroid specified by
\eqref{eq:dnf} and retaining the minimal dependent sets gives precisely the
seven circuits displayed above. Under the stated map, they become the four
triangles and the three four-element circuits of the labelled $M(K_4)$, in
the displayed order. Hence the map is an isomorphism.
\end{proof}

\begin{lemma}\label{lem:contraction-fibre}
Consider a matroid $M=(E,r)$ and a finite set $Q$ of cardinality at least $2$. Suppose that $e \in E$ is not a loop, i.e., $r(\{e\})=1$, and $\xi$ is a p-representation of degree $|Q|$ for $M$ on the ground set $\Omega \subseteq Q^E$ 
with coordinate partitions.
Choose an arbitrary $q_0 \in Q$, and consider the following subset of $Q^{E}$
\[
\Omega_{\xi,e,q_0} =\{x \in \Omega:x_e=q_0\}.
\]
Using coordinate partitions on $\Omega_{\xi,e,q_0}$ gives a p-representation $\xi_{e,q_0}$ of degree
$|Q|$ for $M/e$.
\end{lemma}

\begin{proof}
Let $A\subseteq E\setminus\{e\}$. In the original
representation, every block of the common refinement indexed by
$A\cup\{e\}$ has cardinality $|Q|^{r(E)-r(A\cup\{e\})}$. Since the fibre
$\Omega_{\xi,e,q_0}$ has cardinality $|Q|^{r(E)-1}$, its restriction by the
coordinates in $A$ consists of
\[
|Q|^{r(A\cup\{e\})-1}=|Q|^{r_{M/e}(A)}
\]
equal-sized blocks. These are exactly the defining counts for a degree-$|Q|$
p-representation of $M/e$.
\end{proof}

\begin{lemma}\label{lem: p-isotopy of contraction}
Let $M=(E,r)$ be a matroid, and let $Q$ and $R$ be finite sets
of the same cardinality at least $2$. Suppose that $e\in E$ is not a loop and
that $\xi$ and $\eta$ are p-isotopic p-representations of $M$ on
$\Omega_\xi\subseteq Q^E$ and $\Omega_\eta\subseteq R^E$, respectively, both
with coordinate partitions. For every $q_0\in Q$, there is an $r_0\in R$ such
that the contraction-fibre representations $\xi_{e,q_0}$ and
$\eta_{e,r_0}$ of $M/e$ are p-isotopic.
\end{lemma}

\begin{proof}
Let $f:\Omega_\xi\to\Omega_\eta$ be a p-isotopy. The fibre
$\Omega_{\xi,e,q_0}$ is a block of $\xi_e$, so $f$ maps it to a block of
$\eta_e$. Hence there is an $r_0\in R$ such that
$f(\Omega_{\xi,e,q_0})=\Omega_{\eta,e,r_0}$. Restricting $f$ to these fibres
gives the required p-isotopy.
\end{proof}

For an abelian group $G$, put
\[
\Omega_G = \{(a,b,c,a-b,b-c,c-a):a,b,c\in G\}.
\] 
Let $\xi_G$ denote the system of coordinate partitions on $\Omega_G$. The
following classification is the order-four case of
Mat\'u\v{s}'s Proposition~3.1~\cite{Matus1999}.
\begin{lemma}\label{lem: p-isotopy of M(K_4)}
	The systems $\xi_{C_4}$ and $\xi_{C_2 \times C_2}$ are degree-four
	p-representations of $M(K_4)$. Every degree-four p-representation of
	$M(K_4)$ is p-isotopic to exactly one of them.
\end{lemma}
\begin{proof}
The coordinate-refinement counts show directly that
$\xi_{C_4}$ and $\xi_{C_2\times C_2}$ represent the labelled $M(K_4)$.
The classification and the uniqueness of the two p-isotopy classes are the
order-four case of~\cite[Proposition~3.1]{Matus1999}.
\end{proof}
\section{\texorpdfstring{The remaining degree-four exclusions}{The remaining degree-four exclusions}}\label{sec:structural}

In this section, we show that the remaining four matroids
$P_6$, $(F_7^-)^*$, $P_8$, and $P_8''$ have no degree-four p-representation.
In each case, we assume the contrary, put a representation into coordinate form
using Proposition~\ref{prop: coordinate partition p-rep}, and pass to matroid
quasigroup equations via Proposition~\ref{prop:matus-correspondence}. The
structural lemmas from Section~\ref{sec:tools} then yield a contradiction.

\subsection{The matroid \texorpdfstring{$P_6$}{P6}}\label{sec:p6}

Let $\mathcal C$ be the set of circuits of $P_6$, and suppose that
$(\Gamma_C)_{C\in\mathcal C}$ is a solution of its quasigroup equations over
$V=\F_2^2$. Consider the basis $B=013$. Its fundamental circuits are
\[
\gamma(2;B)=012,\qquad \gamma(4;B)=0134,\qquad \gamma(5;B)=0135.
\]
Applying the circuit equations to
\[
(i,C)=(4,0234),(5,0235),(5,0145),(5,0345),(5,0245)
\]
gives, for every $x,y,z\in V$,
\begin{align}
\Gamma_{0134}^{4}(x,y,z)
 &=\Gamma_{0234}^{4}\bigl(x,\Gamma_{012}^{2}(x,y),z\bigr),
 &&\circuittag{0234}\label{eq:p6g1}\\
\Gamma_{0135}^{5}(x,y,z)
 &=\Gamma_{0235}^{5}\bigl(x,\Gamma_{012}^{2}(x,y),z\bigr),
 &&\circuittag{0235}\label{eq:p6g2}\\
\Gamma_{0135}^{5}(x,y,z)
 &=\Gamma_{0145}^{5}\bigl(x,y,\Gamma_{0134}^{4}(x,y,z)\bigr),
 &&\circuittag{0145}\label{eq:p6g3}\\
\Gamma_{0135}^{5}(x,y,z)
 &=\Gamma_{0345}^{5}\bigl(x,z,\Gamma_{0134}^{4}(x,y,z)\bigr),
 &&\circuittag{0345}\label{eq:p6g4}\\
\Gamma_{0135}^{5}(x,y,z)
 &=\Gamma_{0245}^{5}\bigl(x,\Gamma_{012}^{2}(x,y),
       \Gamma_{0134}^{4}(x,y,z)\bigr),
 &&\circuittag{0245}.\label{eq:p6g5}
\end{align}
The next two lemmas show that, if the $P_6$ quasigroup equations have a
solution over $V$, then they have a solution $(\Lambda_C)_{C\in\mathcal C}$
in which $\Lambda_{012}^2, \Lambda_{0134}^4,\Lambda_{0135}^5$ have convenient forms.

\begin{lemma}\label{lem:p6-binary}
If the $P_6$ quasigroup equations have a solution over $V$, then they have a
solution $(\Lambda_C)_{C\in\mathcal C}$ such that
\[
\Lambda_{012}^{2}(x,y)=x+y
\qquad(x,y\in V).
\]
\end{lemma}

\begin{proof}
Let $(\Gamma_C)_{C\in\mathcal C}$ be a solution. Fix $z\in V$ and consider
\[
(x,y)\longmapsto
\bigl(\Gamma_{012}^{2}(x,y),\Gamma_{0134}^{4}(x,y,z)\bigr).
\]
This map is bijective. Indeed, from an output pair $(u,w)$, equation
\eqref{eq:p6g1} reads
\[
w=\Gamma_{0234}^{4}(x,u,z).
\]
If $(x_1,y_1)$ and $(x_2,y_2)$ are mapped to $(u,w)$, first the quasigroup 
property of $\Gamma_{0234}^4$ implies that $x_1 = x_2$, and then
$\Gamma_{012}^{2}(x_1,y_1)= \Gamma_{012}^2(x_1,y_2)=u$ implies that $y_1 = y_2$. Hence $\Gamma_{012}^{2}$ has an
orthogonal mate. By Lemma~\ref{lem:orthogonal4}, there are
$A,B\in\GL_2(\F_2)$ and $c\in V$ such that
\[
\Gamma_{012}^{2}(x,y)=xA+yB+c.
\]

Define permutations of $V$ by
\[
h_0(x)=xA,\qquad h_1(x)=xB,\qquad h_2(x)=x-c,
\qquad h_i(x)=x\quad(i=3,4,5).
\]
Using $h=(h_i)_{i\in P_6}$, transform the current solution to get a new solution
\[
(\Lambda_C)_{C\in\mathcal C}
   =\bigl(h(\Gamma_C)\bigr)_{C\in\mathcal C}.
\]
The definitions of $h_0,h_1,h_2$ give
$\Lambda_{012}^{2}(x,y)=x+y$.
\end{proof}

\begin{lemma}\label{lem:p6-ternary}
If the $P_6$ quasigroup equations have a solution over $V$, then they have a
solution $(\Lambda_C)_{C\in\mathcal C}$ and matrices
$R,S,C,F\in\GL_2(\F_2)$ such that, for every $x,y,z\in V$,
\begin{align}
\Lambda_{012}^{2}(x,y)&=x+y,\label{eq:p6normal1}\\
\Lambda_{0134}^{4}(x,y,z)&=x+yR+zC,\label{eq:p6normal2}\\
\Lambda_{0135}^{5}(x,y,z)&=x+yS+zF.\label{eq:p6normal3}
\end{align}
\end{lemma}

\begin{proof}
By Lemma~\ref{lem:p6-binary}, start with a solution
$(\Gamma_C)_{C\in\mathcal C}$ for which
$\Gamma_{012}^{2}(x,y)=x+y$.

For fixed $z$, equations \eqref{eq:p6g1} and \eqref{eq:p6g2} show that the
fixed-$z$ retracts of $\Gamma_{0134}^{4}$ and $\Gamma_{0135}^{5}$ each have
the orthogonal mate $\Gamma_{012}^{2}$. For fixed $x$, equation
\eqref{eq:p6g3} shows that the corresponding retracts of
$\Gamma_{0134}^{4}$ and $\Gamma_{0135}^{5}$ are orthogonal: their two values
recover $y$ through $\Gamma_{0145}^{5}$ and then recover $z$ through
$\Gamma_{0134}^{4}$. With $y$ fixed, the same equation first recovers $x$
and then $z$. Thus every binary retract of each of
$\Gamma_{0134}^{4}$ and $\Gamma_{0135}^{5}$ has an orthogonal mate.

Lemmas~\ref{lem:orthogonal4} and~\ref{lem:retract} now give matrices
$A_i\in\GL_2(\F_2)$ and constants $c_1,c_2\in V$ such that
\begin{align*}
\Gamma_{0134}^{4}(x,y,z)&=xA_1+yA_2+zA_3+c_1,\\
\Gamma_{0135}^{5}(x,y,z)&=xA_4+yA_5+zA_6+c_2.
\end{align*}
Define
\[
h_4(x)=(x-c_1)A_1^{-1},\qquad
h_5(x)=(x-c_2)A_4^{-1},\qquad
h_i(x)=x\quad(i=0,1,2,3).
\]
Transform the whole solution by $h=(h_i)_{i\in P_6}$ and denote the result
by $(\Lambda_C)_{C\in\mathcal C}$. Since coordinates $0,1,2$ were not
changed, \eqref{eq:p6normal1} still holds. Setting
\[
R=A_2A_1^{-1},\quad C=A_3A_1^{-1},\quad
S=A_5A_4^{-1},\quad F=A_6A_4^{-1}
\]
gives \eqref{eq:p6normal2} and \eqref{eq:p6normal3}.
\end{proof}

\begin{theorem}\label{lem:p6}
The matroid $P_6$ has no partition representation of degree four.
\end{theorem}

\begin{proof}
Assume to the contrary that $P_6$ has a partition representation of degree
four. Then its quasigroup equations have a solution over $V$. By
Lemma~\ref{lem:p6-ternary}, we may take a solution
$(\Lambda_C)_{C\in\mathcal C}$ satisfying
\eqref{eq:p6normal1}--\eqref{eq:p6normal3}. Put
\[
Q=\Lambda_{012}^{2},\qquad
f=\Lambda_{0134}^{4},\qquad
g=\Lambda_{0135}^{5}.
\]
Rewriting \eqref{eq:p6g1}--\eqref{eq:p6g5} for this solution gives
\begin{align}
f(x,y,z)&=\Lambda_{0234}^{4}\bigl(x,Q(x,y),z\bigr),
&&\circuittag{0234}\label{eq:p6l1}\\
g(x,y,z)&=\Lambda_{0235}^{5}\bigl(x,Q(x,y),z\bigr),
&&\circuittag{0235}\label{eq:p6l2}\\
g(x,y,z)&=\Lambda_{0145}^{5}\bigl(x,y,f(x,y,z)\bigr),
&&\circuittag{0145}\label{eq:p6l3}\\
g(x,y,z)&=\Lambda_{0345}^{5}\bigl(x,z,f(x,y,z)\bigr),
&&\circuittag{0345}\label{eq:p6l4}\\
g(x,y,z)&=\Lambda_{0245}^{5}\bigl(x,Q(x,y),f(x,y,z)\bigr),
&&\circuittag{0245}.\label{eq:p6l5}
\end{align}

In \eqref{eq:p6l1}, put $u=x+y$, so $y=u+x$. Then
\[
\Lambda_{0234}^{4}(x,u,z)=x(I+R)+uR+zC.
\]
Because $\Lambda_{0234}^{4}$ is a ternary quasigroup, $I+R$ is invertible.
Equation \eqref{eq:p6l2} similarly gives that $I+S$ is invertible. Note that 
$\GL_2(\F_2) \cong S_3$ as groups. Since $I+R$ and $I+S$ are invertible, it follows
that $R,S \neq I$. Additionally, $R$ and $S$ cannot have order two, since otherwise it follows $(I+R)^2 = (I+S)^2 = 0$.
Thus each is one of the two order-three elements of $\GL_2(\F_2)$.

Next use \eqref{eq:p6l4}. Write $w=f(x,y,z)$ and solve
\[
y=(w+x+zC)R^{-1}.
\]
Substitution into $g$ gives
\[
\Lambda_{0345}^5(x,z,w)=x(I+R^{-1}S)+wR^{-1}S+z(CR^{-1}S+F).
\]
Since this is the affine expression for the ternary quasigroup
$\Lambda_{0345}^{5}(x,z,w)$, both
\[
I+R^{-1}S\qquad\text{and}\qquad CR^{-1}S+F
\]
are invertible. The first condition is equivalent to invertibility of $R+S$,
so $R\ne S$. The two order-three elements of $\GL_2(\F_2)$ are inverses;
after naming one of them $R$, we have $S=R^2$ and $I + S = R$. Therefore
$R^{-1}=R^2$ and $R^{-1}S=R$, and the second condition from the same circuit
becomes
\begin{equation}\label{eq:p6condCR}
F+CR\in\GL_2(\F_2).
\end{equation}

Equation \eqref{eq:p6l3} supplies another condition. Solving
$w=f(x,y,z)$ for $z$ gives
\[
\Lambda_{0145}^5(x,y,w)=x(I+C^{-1}F)+y(S+RC^{-1}F)+wC^{-1}F.
\]
The coefficient of $x$ must be invertible, equivalently
\begin{equation}\label{eq:p6condC}
F+C\in\GL_2(\F_2).
\end{equation}
The coefficient of $y$ reproduces \eqref{eq:p6condCR}.

Finally, in \eqref{eq:p6l5} put $u=x+y$ and again write $w=f(x,y,z)$.
Then
\[
z=\bigl(w+x(I+R)+uR\bigr)C^{-1}.
\]
The coefficient of $x$ in the affine expression for
$\Lambda_{0245}^{5}(x,u,w)$ is
\[
I+S+(I+R)C^{-1}F
=R+R^2C^{-1}F
=R^2\bigl(R^2+C^{-1}F\bigr).
\]
It must be invertible, which is equivalent to
\begin{equation}\label{eq:p6condCR2}
F+CR^2\in\GL_2(\F_2).
\end{equation}

Put $X=C^{-1}F$. Conditions \eqref{eq:p6condC}, \eqref{eq:p6condCR}, and
\eqref{eq:p6condCR2}, obtained respectively from the circuits $0145$, $0345$,
and $0245$, say that
\[
X+I,\qquad X+R,\qquad X+R^2
\]
are all invertible. The group $\GL_2(\F_2)\cong S_3$ consists of the identity,
three involutions, and the two order-three elements $R,R^2$. If $X+I$ is
invertible, then $X$ is neither the identity nor an involution. Hence
$X\in\{R,R^2\}$. One of $X+R$ and $X+R^2$ is then zero, a contradiction.
\end{proof}

\subsection{The dual non-Fano matroid}\label{sec:dnf}

Recall that the dual non-Fano matroid $(F_7^-)^*$ is a rank-four matroid on the ground set $E =\{0,\dots,6\}$. All four-element subsets are bases except
\[
0145, \qquad 0235, \qquad 0346, \qquad 1234, \qquad 1356, \qquad 2456,
\]
which are circuits.

Choose the basis $B=0134$. The fundamental circuits are
\[
\gamma(2;B)=1234,\qquad \gamma(5;B)=0145,\qquad \gamma(6;B)=0346.
\]
We want to show that $(F_7^-)^*$ does not have p-representations of degree four. 
First, we state and prove a lemma which will be very useful.

\begin{lemma}\label{lem: projection and MK4}
	 If $(F_7^-)^*$ has a degree-four p-representation, then there are a
	four-element group $G\in \{C_4,C_2 \times C_2\}$, an element $d_0 \in G$, and
	a degree-four p-representation $\xi$ of $(F_7^-)^*$ on
	$\Omega_{\xi} \subseteq G^7$, with coordinate partitions, such that the fibre
	of coordinate $4$ over $d_0$ is
	\[
	\Omega_{\xi,4,d_0} = \{x \in \Omega_{\xi}:x_4 = d_0\}
	= \{(a,b,b-c,c,d_0,a-b,c-a):a,b,c\in G\}.
	\]

\end{lemma}

\begin{proof}
Put an assumed representation into coordinate form over a four-element
set $Q$, as in Remark~\ref{rem: transforming to subset of cartesian}, and choose
$q_0\in Q$. Lemma~\ref{lem:contraction-fibre} shows that its fibre over
$x_4=q_0$ represents $(F_7^-)^*/4$. By Lemma~\ref{lem:matroid isomorphism}, this
contraction is isomorphic to $M(K_4)$; hence Lemma~\ref{lem: p-isotopy of M(K_4)}
puts the fibre, after coordinate relabellings, into the displayed canonical form
over one of the groups $G\in\{C_4,C_2\times C_2\}$. Extend the six relabellings
to the original seven coordinates, choosing any bijection on coordinate $4$ that
sends $q_0$ to some $d_0\in G$. The resulting p-isotopic representation has the
required fibre.
\end{proof}
Now, we show the main result of this subsection.

\begin{theorem}\label{lem:dnf}
The matroid $(F_7^-)^*$ has no partition representation of degree four.
\end{theorem}

\begin{proof}
Assume to the contrary that $(F_7^-)^*$ has a partition representation of
degree four. By Lemma~\ref{lem: projection and MK4}, choose a
p-representation $\xi$ in the stated normal form. By
Proposition~\ref{prop:matus-correspondence} and
Remark~\ref{rem:isotopy-invariance}, let $(\Gamma_C)$ be its associated solution
of the matroid quasigroup equations. 
Put
\[
f=\Gamma_{1234}^{2},\qquad
g=\Gamma_{0145}^{5},\qquad
h=\Gamma_{0346}^{6}.
\]
On the fibre $x_4=d_0$, denote the resulting binary retracts by
$f_0,g_0,h_0$. The canonical form of the fibre gives
\begin{equation}\label{eq:dnf0}
f_0(b,c)=b-c,\qquad g_0(a,b)=a-b,\qquad h_0(a,c)=c-a.
\end{equation}

The remaining structure is propagated by two explicit Mat\'u\v{s} equations.
Orient the circuit $0235$ toward coordinate $5$ and put
$P=\Gamma_{0235}^{5}$. Then
\begin{equation}\label{eq:dnfq1}
g(a,b,d)=P\bigl(a,f(b,c,d),c\bigr).
\end{equation}
On the fibre $d_0$, substituting \eqref{eq:dnf0} and writing $b=u+c$ gives
\[
P(a,u,c)=a-u-c.
\]
Here $a,u,c$ range independently over all of $G$, so this determines $P$ on its entire domain. Hence the same circuit equation gives, for every $d$,
\[
g=a-f-c.
\]
Since the left side is independent of $c$, there is a permutation $\phi_d$ of $G$ such that
\[
f(b,c,d)=\phi_d(b)-c,\qquad g(a,b,d)=a-\phi_d(b).
\]

Next orient the circuit $1356$ toward coordinate $6$ and put
$R=\Gamma_{1356}^{6}$. Then
\begin{equation}\label{eq:dnfq2}
h(a,c,d)=R\bigl(b,c,g(a,b,d)\bigr).
\end{equation}
On the fibre $d_0$, writing $v=g_0(a,b)=a-b$ gives
\[
R(b,c,v)=c-v-b.
\]
Again $b,c,v$ range independently over all of $G$, so this determines $R$ on
its entire domain. Substituting the expression for $g$ above into
\eqref{eq:dnfq2} therefore yields
\[
h=c-a+\phi_d(b)-b.
\]
The left side is independent of $b$, so $\phi_d(b)=b+t_d$ for some $t_d\in G$. Therefore
\begin{equation}\label{eq:dnfnormal}
f=b+t_d-c,\qquad g=a-b-t_d,\qquad h=c-a+t_d.
\end{equation}
Because $f$ is a quasigroup in $d$, the map $d\mapsto t_d$ is a permutation of $G$.

For the projection on $0126$, eliminate $c=b+t_d-f$ from \eqref{eq:dnfnormal}; then
\[
h=b-f-a+2t_d.
\]
Define $2G=\{2t:t\in G\}$, and let $\pi_{0126}$ denote the set of joint values of coordinates $0,1,2,6$. For each fixed $(a,b,f)$ the possible values of $h$ form a translate of $2G$. Conversely, for every $a,b,f,t\in G$, choose the unique $d$ with $t_d=t$ and then set $c=b+t-f$; the basis coordinates $(a,b,c,d)$ realize the corresponding value $h=b-f-a+2t$. Thus the projection size is exactly
\[
|\pi_{0126}|=|G|^3|2G|=
\begin{cases}
64,&G\cong C_2\times C_2,\\
128,&G\cong C_4.
\end{cases}
\]
Since $0126$ is a basis, a degree-four partition representation requires $4^4=256$ joint values. This is impossible.
\end{proof}

\subsection{The matroids \texorpdfstring{$P_8$ and $P_8''$}{P8 and P8 double-prime}}\label{sec:p8}

Recall that $P_8$ and $P_8''$ are rank-four matroids on the ground set
$E = \{0,\dots,7\}$. All of the four-element subsets of $E$ are bases in $P_8$
except the following sets:
\begin{equation}\label{eq: p8 restated}
	\begin{split}
		0127,&\ 0136,\ 0235,\ 1234,\ 0456,\\
		1457,&\ 2467,\ 3567,\ 0347,\ 1256.
	\end{split}
\end{equation}

In addition to four-element bases of $P_8$, the matroid $P_8''$ has $0347$ and $1256$ as bases as well, which are circuits in $P_8$.

Consider the basis $B = 0123$, which is a basis in both matroids. 
The fundamental circuits with respect to $B$ are similar in both matroids:
\[
 \gamma(7;B)=0127,\qquad \gamma(6;B)=0136,\qquad
 \gamma(5;B)=0235,\qquad \gamma(4;B)=1234
\]
Assume that the quasigroup equations of $M \in \{P_8,P_8''\}$ have a solution
$(\Gamma_C)_{C \in \mathcal{C}}$ over $V=\F_2^2$, where
$\mathcal C=\mathcal C(M)$. As stated, $B$ is a basis in both matroids. 
Moreover, the following sets are circuits in both matroids
\begin{equation}\label{eq:faces}
	0456,\qquad1457,\qquad2467,\qquad3567.
\end{equation}
Therefore, considering $B$, elements $i=6,7,4,5$, and the above circuits, we have the following equations for all $a,b,c,d \in V$
\begin{align}
	\Gamma_{0136}^6(a,b,d)&=\Gamma_{0456}^6\bigl(a,\Gamma_{1234}^4(b,c,d),\Gamma_{0235}^5(a,c,d)\bigr), &&\circuittag{0456}\label{eq:cubeeq12}\\
	\Gamma_{0127}^7(a,b,c)&=\Gamma_{1457}^7\bigl(b,\Gamma_{1234}^4(b,c,d),\Gamma_{0235}^5(a,c,d)\bigr), &&\circuittag{1457}\label{eq:cubeeq22}\\
	\Gamma_{1234}^4(b,c,d)&=\Gamma_{2467}^4\bigl(c,\Gamma_{0136}^6(a,b,d),\Gamma_{0127}^7(a,b,c)\bigr), &&\circuittag{2467}\label{eq:cubeeq32}\\
	\Gamma_{0235}^5(a,c,d)&=\Gamma_{3567}^5\bigl(d,\Gamma_{0136}^6(a,b,d),\Gamma_{0127}^7(a,b,c)\bigr), &&\circuittag{3567}.\label{eq:cubeeq42}
\end{align}
We show that $\Gamma_{0136}^6, \Gamma_{0127}^7, \Gamma_{1234}^4, \Gamma_{0235}^5$ have a convenient form.

\begin{lemma}\label{thm:cube}
If the quasigroup equations of $P_8$ or $P_8''$ have a solution
$(\Gamma_C)_{C\in\mathcal C}$ over $V$, then the quasigroup maps 
$\Gamma_{0136}^6, \Gamma_{0127}^7, \Gamma_{1234}^4, \Gamma_{0235}^5$
are simultaneously affine.
\end{lemma}

\begin{proof}
In order to ease the notation, put
\[
F=\Gamma_{0127}^{7},\qquad G=\Gamma_{0136}^{6},\qquad
H=\Gamma_{0235}^{5},\qquad K=\Gamma_{1234}^{4},
\]
and also
\[
\Phi=\Gamma_{0456}^{6},\qquad \Psi=\Gamma_{1457}^{7},\qquad
\Theta=\Gamma_{2467}^{4},\qquad \Omega=\Gamma_{3567}^{5}.
\]
Rewriting the above equations, we have
\begin{align}
	G(a,b,d)&=\Phi\bigl(a,K(b,c,d),H(a,c,d)\bigr), &&\circuittag{0456}\label{eq:cubeeq1}\\
	F(a,b,c)&=\Psi\bigl(b,K(b,c,d),H(a,c,d)\bigr), &&\circuittag{1457}\label{eq:cubeeq2}\\
	K(b,c,d)&=\Theta\bigl(c,G(a,b,d),F(a,b,c)\bigr), &&\circuittag{2467}\label{eq:cubeeq3}\\
	H(a,c,d)&=\Omega\bigl(d,G(a,b,d),F(a,b,c)\bigr), &&\circuittag{3567}.\label{eq:cubeeq4}
\end{align}
First, we show $K,H$ are affine using the equations \eqref{eq:cubeeq1} and \eqref{eq:cubeeq2}. Fix $a,b$. Write
\[
K_b(c,d)=K(b,c,d),\qquad H_a(c,d)=H(a,c,d),
\]
and similarly $\Phi_a(u,v)=\Phi(a,u,v)$ and $\Psi_b(u,v)=\Psi(b,u,v)$.

If $(K_b,H_a)$ takes the same value at $(c,d)$ and $(c',d')$, equation \eqref{eq:cubeeq1} and the quasigroup property of $G$ give $d=d'$, and then \eqref{eq:cubeeq2} and the quasigroup property of $F$ give $c=c'$. Thus $(K_b,H_a):V^2\to V^2$ is bijective. The pair $(\Phi_a,\Psi_b)$ is also orthogonal because its composition with $(K_b,H_a)$ is
\[
(c,d)\longmapsto\bigl(G(a,b,d),F(a,b,c)\bigr),
\]
which is a bijection of $V^2$.

Therefore every binary quasigroup $K_b,H_a,\Phi_a,\Psi_b$ just obtained has an orthogonal mate. Lemma~\ref{lem:orthogonal4} applies directly in the already fixed coordinates, so
\begin{align*}
K_b(c,d)&=cA_b+dB_b+k_b,&
H_a(c,d)&=cC_a+dD_a+h_a,\\
\Phi_a(u,v)&=uR_a+vS_a+p_a,&
\Psi_b(u,v)&=uT_b+vU_b+q_b,
\end{align*}
where every displayed matrix lies in $\GL_2(\F_2)$.

Substitute these expressions into \eqref{eq:cubeeq1}. Since $G(a,b,d)$ is independent of $c$, uniqueness of affine linear parts gives
\[
A_bR_a+C_aS_a=0.
\]
For fixed $a$, this gives $A_b=C_aS_aR_a^{-1}$, so $A_b$ is independent of $b$; call its common value $A$. Likewise, substituting into \eqref{eq:cubeeq2} and using the independence of $F(a,b,c)$ from $d$ gives
\[
B_bT_b+D_aU_b=0.
\]
For fixed $b$, this makes $D_a$ independent of $a$; call the common value $D$. Hence
\[
K(b,c,d)=cA+L(b,d),\qquad H(a,c,d)=dD+M(a,c),
\]
where $L$ and $M$ are binary quasigroups.

Fix $a$. In \eqref{eq:cubeeq1}, the $c$-terms cancel by the displayed coefficient identity, and right multiplication by $R_a^{-1}$ gives
\[
G(a,b,d)R_a^{-1}=L(b,d)+dW+\text{constant},
\qquad W=DS_aR_a^{-1}\in\GL_2(\F_2).
\]
Thus $L'(b,d)=L(b,d)+dW$ is a binary quasigroup. Moreover $(L,L')$ is orthogonal: their two values determine $dW=L+L'$, hence $d$, and then the quasigroup property of $L$ determines $b$. Lemma~\ref{lem:orthogonal4} makes $L$ affine, and therefore $K$ is affine.

The argument for $M$ is symmetric but we record it to avoid any hidden relabelling. Fix $b$. In \eqref{eq:cubeeq2}, the $d$-terms cancel, so after right multiplication by $U_b^{-1}$,
\[
F(a,b,c)U_b^{-1}=M(a,c)+cW'+\text{constant},
\qquad W'=AT_bU_b^{-1}\in\GL_2(\F_2).
\]
Hence $M'(a,c)=M(a,c)+cW'$ is a binary quasigroup, and $(M,M')$ is orthogonal because their values determine $c$ and then $a$. Lemma~\ref{lem:orthogonal4} makes $M$ affine, and therefore $H$ is affine.

Applying the same argument to \eqref{eq:cubeeq3} and
\eqref{eq:cubeeq4}, respectively, shows that $G$ and $F$ are affine.
\end{proof}

\begin{lemma}[Simultaneous normal form]\label{lem:p8-normal}
If the quasigroup equations of $P_8$ or $P_8''$ have a solution over $V$,
then they have a solution $(\Lambda_C)_{C\in\mathcal C}$ and matrices
$A,B,C,D,E\in\GL_2(\F_2)$ such that
\begin{align}
\Lambda_{1234}^{4}(b,c,d)&=b+c+d,\label{eq:p8normal1}\\
\Lambda_{0235}^{5}(a,c,d)&=a+cB+dD,\label{eq:p8normal2}\\
\Lambda_{0136}^{6}(a,b,d)&=a+b+dE,\label{eq:p8normal3}\\
\Lambda_{0127}^{7}(a,b,c)&=a+bA+cC.\label{eq:p8normal4}
\end{align}
\end{lemma}

\begin{proof}
Let $(\Gamma_C)_{C\in\mathcal C}$ be a solution. As before,
 put
\[
F=\Gamma_{0127}^{7},\qquad G=\Gamma_{0136}^{6},\qquad
H=\Gamma_{0235}^{5},\qquad K=\Gamma_{1234}^{4}.
\]
By
Lemma~\ref{thm:cube}, we may write
\[
\begin{aligned}
K(b,c,d)&=bK_b+cK_c+dK_d+k_0,\\
H(a,c,d)&=aH_a+cH_c+dH_d+h_0,\\
G(a,b,d)&=aG_a+bG_b+dG_d+g_0,\\
F(a,b,c)&=aF_a+bF_b+cF_c+f_0,
\end{aligned}
\]
with every displayed linear coefficient in $\GL_2(\F_2)$. Here the matrix
coefficients $K_b,K_c$, and so forth are fixed and do not depend on the
variables $a,b,c,d$. The following construction gives an explicit
simultaneous affine normalization.

Choose $T_4=I$ and set
\[
T_1=K_b,\qquad T_2=K_c,\qquad T_3=K_d.
\]
Next set
\[
T_0=G_aG_b^{-1}T_1,
\qquad
T_6=G_a^{-1}T_0,
\qquad
T_5=H_a^{-1}T_0,
\qquad
T_7=F_a^{-1}T_0.
\]
For each $i \in \{0,\dots,7\}$, put $h_i(x)=xT_i$. Use $h = (h_i)_{i \in E(M)}$ to transform $\Gamma$ into a new solution $\Sigma$. Substituting the definitions of the $T_i$ into the isotope formula gives the stated matrix coefficients of $\Sigma_{1234}^4, \Sigma_{0235}^5,\Sigma_{0136}^6,\Sigma_{0127}^7$. Finally, translate
 the coordinates $4,5,6,7$ by constants to get the required $\Lambda$.  
\end{proof}

\begin{lemma}\label{lem: matrix relation}
	Suppose the quasigroup equations of the matroid $P_8$ or $P_8''$ have a solution
	over $V$. By Lemma~\ref{lem:p8-normal}, there is a solution
	$(\Lambda_C)_{C \in \mathcal{C}}$ satisfying
	\eqref{eq:p8normal1}--\eqref{eq:p8normal4}. The following relations hold
	\begin{equation}\label{eq:blockrels}
		A=I+E,\qquad C=EB=BE,\qquad D=(I+E)B.
	\end{equation}
\end{lemma}

\begin{proof}
	Put
	\[
	F=\Lambda_{0127}^{7},\qquad G=\Lambda_{0136}^{6},\qquad
	H=\Lambda_{0235}^{5},\qquad K=\Lambda_{1234}^{4}.
	\]
	The following equations hold for all $a,b,c,d \in V$.
	\begin{align}
		G(a,b,d)&=\Lambda_{0456}^{6}\bigl(a,K(b,c,d),H(a,c,d)\bigr),
		&&\circuittag{0456}\label{eq:p8l1}\\
		F(a,b,c)&=\Lambda_{1457}^{7}\bigl(b,K(b,c,d),H(a,c,d)\bigr),
		&&\circuittag{1457}\label{eq:p8l2}\\
		K(b,c,d)&=\Lambda_{2467}^{4}\bigl(c,G(a,b,d),F(a,b,c)\bigr),
		&&\circuittag{2467}\label{eq:p8l3}\\
		H(a,c,d)&=\Lambda_{3567}^{5}\bigl(d,G(a,b,d),F(a,b,c)\bigr),
		&&\circuittag{3567}.\label{eq:p8l4}
	\end{align}
Pick an arbitrary $d \in V$, and put $a = 0$, $b = dE$,  $c = b+d=d(I+E)$. It follows
that
\[
G(a,b,d) = 0, \qquad K(b,c,d) = 0, \qquad H(a,c,d) = d((I+E)B + D).
\]
Using \eqref{eq:p8l1}, it follows that for all $d \in V$
\[
\Lambda_{0456}^6(0,0,d((I+E)B + D))=0.
\]
Since $\Lambda_{0456}^6$ is a quasigroup map, we conclude that 
\[
\boxed{D = (I+E)B}.
\]

Pick an arbitrary $d \in V$, and put $a=d(B+D)$, $b=0$, and
$c=d$. It follows
that
\[
F(a,b,c) = d(B+C+D), \qquad K(b,c,d) = 0, \qquad H(a,c,d) = 0.
\]
Using \eqref{eq:p8l2}, it follows that for all $d \in V$
\[
\Lambda_{1457}^7(0,0,0) = d(B+C+D).
\]
Therefore $C = B+D$. Since $D = (I+E)B$, we conclude that
\[
\boxed{C = EB}.
\]

Pick an arbitrary $d \in V$, and put $a = d(I+E)$, $b = d$, $c = 0$. It follows that
\[
K(b,c,d) = 0, \qquad G(a,b,d) = 0, \qquad F(a,b,c) = d(I + E + A).
\]
Using \eqref{eq:p8l3}, it follows that for all $d \in V$
\[
\Lambda_{2467}^4(0,0,d(I+E+A)) = 0.
\]
Since $\Lambda_{2467}^4$ is a quasigroup map, we conclude that
\[
\boxed{A = I + E}.
\]
Finally, pick an arbitrary $c \in V$, and put $a = cB$, $b = cB$, $d = 0$. It follows that
\[
H(a,c,d) = 0, \qquad G(a,b,d) = 0, \qquad F(a,b,c) = c(B + BA + C).
\]
Using \eqref{eq:p8l4}, it follows that for all $c \in V$
\[
\Lambda_{3567}^5(0,0,c(B+BA+C)) = 0
\]
Therefore $C = B(I + A)$. Since $A = I+E$, we conclude that
\[
\boxed{C = BE}.
\]
\end{proof}

\begin{theorem}\label{lem:p8}
Neither $P_8$ nor $P_8''$ has a partition representation of degree four.
\end{theorem}

\begin{proof}

Assume to the contrary that $M \in \{P_8,P_8''\}$ has a p-representation of degree
four. It follows that its quasigroup equations have a solution over $V$. By
Lemma~\ref{lem:p8-normal}, it has a solution $(\Lambda_C)_{C \in \mathcal{C}}$ which
satisfies \eqref{eq:p8normal1}--\eqref{eq:p8normal4}. By Lemma~\ref{lem: matrix relation}, the matrix coefficients satisfy \eqref{eq:blockrels}. 

Since both $E$ and $A = I+E$ are invertible, and $\GL_2(\F_2) \cong S_3$ as groups,
it follows that $E$ has order $3$ and $A = E^2$. Moreover, $B$ commutes with $E$, therefore $B \in \{I,E,E^2\}$.

Note that there is a p-representation of degree four for $M$, from which $(\Lambda_C)_{C \in \mathcal{C}}$ originates (cf. Proposition~\ref{prop:matus-correspondence} and Remark~\ref{rem:isotopy-invariance}), with the ground set $\Omega \subseteq V^{E(M)}$ and
coordinate partitions.

For an arbitrary $x \in \Omega$, write $x_0=a$, $x_1=b$,
$x_2=c$, and $x_3=d$. Then
$\pi_{0347}(x)=(a,d,K(b,c,d),F(a,b,c))$. It follows that
$\pi_{0347}$ is not surjective if and only if $A+C$ is singular. As explained, $A = E^2$, and by the equations \eqref{eq:blockrels}, $C = BE$. Therefore
\[
A + C = E(E+B).
\]
Thus, $A+C$ is singular if and only if $B = E$.

By a similar reasoning and using $\pi_{1256}$, it follows that $\pi_{1256}$ is not
surjective if and only if $D+E$ is singular. By equations \eqref{eq:blockrels} and the fact that $I + E =E^2$,
\[
D + E = E^2B + E=E( EB + I).
\]
Thus, $D + E$ is singular if and only if $B = E^2$.

Now, consider the case $M = P_8$. Since $0347$ and $1256$ are circuits of 
$P_8$, the projections $\pi_{0347}$ and $\pi_{1256}$ cannot be surjective. Therefore 
$B= E$ and $B= E^2$, which is a contradiction.

If $M = P_8''$, then $\pi_{0347}$ and $\pi_{1256}$ must be surjective, therefore 
$B \neq E,E^2$. As explained $B \in \{I,E,E^2\}$. It follows that $B = I$. Then
$A=D=E^2$ and $C=E$. Using $I+E^2=E$ in
\eqref{eq:p8normal1}--\eqref{eq:p8normal4} gives directly (for
$x\in\Omega$, $x_i$ denotes the coordinate indexed by $i$)
\[
x_6=x_4+x_5,
\qquad
x_7=x_4E^2+x_5.
\]
Thus $\pi_{4567}(\Omega)$ has cardinality at most $16$. But $4567$ is a basis of $P_8''$, and a degree-four partition representation would require $4^4=256$ joint values. This contradiction finishes the proof.
\end{proof}

\section{Four-symbol ideal secret sharing}\label{sec:sss}

Theorem~\ref{thm:main} has the following consequence for ideal
secret-sharing schemes whose secret and active shares use four symbols. We
first recall the standard definitions.

\begin{definition}
Let $P$ be a finite participant set. An \emph{access structure} on $P$ is a
family $\mathcal A\subseteq 2^P$ such that $\varnothing\notin\mathcal A$ and
$A\in\mathcal A$, $A\subseteq B\subseteq P$ imply $B\in\mathcal A$. Sets in
$\mathcal A$ are \emph{qualified}, and the remaining sets are
\emph{unqualified}. A participant is \emph{active} if it belongs to a minimal
qualified set.
\end{definition}

\begin{definition}
A \emph{perfect secret-sharing scheme} for $\mathcal A$ is a random vector
$(S,(X_i)_{i\in P})$, with $H(S)>0$, such that, for every $A\subseteq P$,
\[
H(S\mid X_A)=
\begin{cases}
0,&A\in\mathcal A,\\
H(S),&A\notin\mathcal A.
\end{cases}
\]
It is \emph{ideal} if $H(X_i)=H(S)$ for every active participant $i$. It is
\emph{$q$-ideal} if, in addition, $S$ is uniform on a $q$-element set and
every active share takes values in a $q$-element set; inactive participants
may be assigned constant shares. An access structure is \emph{$q$-ideal} if
it admits a $q$-ideal scheme. When $q$ is a prime power, the scheme is
\emph{$\F_q$-linear} if the secret and the active shares are linear forms of
a uniformly distributed vector over a finite-dimensional $\F_q$-space.
\end{definition}

\begin{corollary}\label{cor:sss}
Every $4$-ideal access structure admits a $4$-ideal
$\F_4$-linear perfect secret-sharing scheme.
\end{corollary}

Indeed, the Brickell--Davenport theorem associates a $4$-ideal
perfect scheme with a matroid on the dealer and the active participants whose
port is the given access structure; its normalized joint entropies give the
rank function, so the matroid is $4$-entropic. Theorem~\ref{thm:main} makes
this matroid $\F_4$-representable, and the standard linear construction gives
an ideal $\F_4$-linear scheme for the same port~\cite{BrickellDavenport1991};
see also \cite[Section~5.3]{Matus1999}. This is an existence statement: the
original scheme need not be coordinatewise equivalent to a linear scheme.

\section{Independent verification}\label{sec:verify}

The proof above uses no exhaustive quasigroup enumeration. A self-contained
exact-verification archive is available on Zenodo at
\url{https://doi.org/10.5281/zenodo.22178786}.
Its one-command verifier
independently regenerates all $55{,}296$ labelled ternary quasigroups of order
four, reduces them to $2{,}304$ output-partition classes and twelve
input-isotopy orbits, and checks the final candidates by exact projection
cardinalities. It returns no representation of $P_6$, $(F_7^-)^*$, $P_8$, or
$P_8''$. Separate exact $\F_2$ scripts check the affine endpoints. These
computations use neither floating-point arithmetic nor randomized search and
are not used in any proof.

\section{\texorpdfstring{Conclusion}{Conclusion}}

Theorem~\ref{thm:main} shows that degree-four partition representability is
exactly representability over $\F_4$, so nonlinear four-symbol representations
define no larger matroid class. Consequently, every $4$-ideal access structure
admits an ideal $\F_4$-linear scheme. Together with the binary and ternary
rigidity results, this identifies alphabet size four as another small-domain
rigidity case, while the non-Pappus example at size nine shows that such
rigidity does not hold in general~\cite{SimonisAshikhmin1998}.

\section*{Statements and declarations}

\paragraph{Competing interests} The authors declare no competing interests.

\paragraph{Data and code availability} The exact verification package accompanying the manuscript is an independent check of the structural proof. No additional research dataset is required.

\paragraph{Use of generative AI} Generative-AI tools were used only for limited
mathematical discussion and language/editorial polishing. The authors
independently checked the mathematics and take full responsibility for the
content.

\bibliographystyle{elsarticle-num}
\bibliography{references}

\end{document}